\documentclass[reqno,centertags,12pt]{amsart}
\usepackage[letterpaper,hmargin=1.5in,top=1.5in,bottom=1.5in]{geometry}
\usepackage{amsmath,amsthm,amssymb,enumerate,esint,amscd,latexsym}
\usepackage[bookmarksopen=true]{hyperref}

\newtheorem{theorem}{Theorem}

\theoremstyle{definition}

\theoremstyle{remark}

\numberwithin{equation}{section}
\numberwithin{theorem}{section}
\newcounter{smalllist}

\newcommand{\bi}{\bibitem}
\newcommand{\no}{\notag}
\newcommand{\lb}{\label}
\newcommand{\f}{\frac}
\newcommand{\bb}{\mathbb}
\newcommand{\cc}{\mathcal}

\newcommand{\pd}{\partial}
\newcommand{\Om}{\Omega}

\newcommand{\si}{\sigma}
\newcommand{\la}{\lambda}

\newcommand{\al}{\alpha}
\newcommand{\be}{\beta}

\newcommand{\ga}{\gamma}

\newcommand{\de}{\delta}

\newcommand{\ze}{\zeta}

\newcommand{\eps}{\varepsilon}
\newcommand{\E}{\mathsf{E}}

\newcommand{\tr}{\text{\rm{tr}}}
\newcommand{\dist}{\text{\rm{dist}}}

\newcommand{\dd}{{d}}

\renewcommand{\Re}{\text{\rm Re}}
\renewcommand{\Im}{\text{\rm Im}}
\newcommand{\<}{\langle}
\renewcommand{\>}{\rangle}

\newcommand{\bsn}{\raisebox{1pt}{$\smallsetminus$}}

\begin{document}

\title[Lieb--Thirring Inequalities for Complex Jacobi Matrices]{Lieb--Thirring Inequalities for Complex Finite Gap Jacobi Matrices}

\author[J.~S.~Christiansen]{Jacob S.~Christiansen}
\address{
Centre for Mathematical Sciences\\
Lund University, Box 118\\
SE-22100, Lund, Sweden.}
\email{stordal@maths.lth.se}

\author[M.~Zinchenko]{Maxim Zinchenko}
\address{
Department of Mathematics and Statistics\\
University of New Mexico\\
Albuquerque, NM 87131.}
\email{maxim@math.unm.edu}

\thanks{JSC is supported in part by the Research Project Grant DFF--4181-00502 from the Danish Council for Independent Research.}
\thanks{MZ is supported in part by Simons Foundation Grant CGM-281971.}


\date{\today}

\begin{abstract}
We establish Lieb--Thirring power bounds on discrete eigenvalues of Jacobi operators for Schatten class complex perturbations of periodic
and more generally finite gap almost periodic Jacobi matrices.
\end{abstract}

\maketitle

\section{Introduction}

In this paper we consider bounded non-selfadjoint Jacobi operators on $\ell^2(\bb Z)$ represented by tridiagonal matrices
\begin{align}
J=\begin{pmatrix}
\ddots&\ddots&\ddots\\
& a_0 & b_1 & c_1 &     &\\
&     & a_1 & b_2 & c_2 &\\
&     &     & a_2 & b_3 & c_3\\
&&&&\ddots&\ddots&\ddots
\end{pmatrix}
\end{align}
with bounded complex parameters $\{a_n,b_n,c_n\}_{n\in\bb Z}$. Our goal is to obtain Lieb--Thirring inequalities for complex perturbations of periodic and, more generally, almost periodic Jacobi operators with absolutely continuous finite gap spectrum.

Lieb--Thirring inequalities for selfadjoint and complex perturbations of the discrete Laplacian have been studied extensively in the last decade \cite{BGK09,DHK13,GK07,Ha11,HK11,HS02}. The original work of Lieb and Thirring \cite{LT75,LT76} was done in the context of continuous Schr\"odinger operators, motivated by their study of the stability of matter. We refer to \cite{DHK09,FLLS06,FSW08,FS,F,HLT98,We96} for more recent developments on Lieb--Thirring inequalities for Schr\"odinger operators and to \cite{HM07} for a review and history of the subject.

Much less is known for perturbations (especially complex ones) of operators with gapped spectrum. Lieb--Thirring inequalities for selfadjoint perturbations of
periodic and almost periodic Jacobi operators with absolutely continuous finite gap spectrum have been established only recently \cite{CZ16,DKS10,FS11,HS08}.
Analogs of these finite gap Lieb--Thirring inequalities for complex perturbations are not known. The aim of the present work is to fill this gap. What is currently
known in the case of complex perturbations is the closely related class of Kato inequalities \cite{GK12,Ha13}. Such inequalities have larger exponents on the eigenvalue side when compared to Lieb--Thirring inequalities (cf.\ \eqref{Kato-SA} vs.\ \eqref{LT-SA} and \eqref{Kato-NSA} vs.\ \eqref{LT-NSA}) and hence are not optimal for small perturbations of Jacobi operators.


To put our new results in perspective, we first discuss the best currently known results on eigenvalue power bounds for Jacobi operators in more detail. The spectral theory
for perturbations of the {\it free} Jacobi matrix, $J_0$, (i.e., the case of $a_n=c_n\equiv 1$ and $b_n\equiv 0$) is well-understood, see \cite{Si11}.
Let $\E=\si(J_0)=[-2,2]$ and suppose $J$ is a selfadjoint Jacobi matrix (i.e., $a_n=c_n>0$) such that $\de J=J-J_0$ is a compact operator, that is, $J$ is a compact
selfadjoint perturbation of $J_0$. Hundertmark and Simon \cite{HS02} proved the following Lieb--Thirring inequalities,
\begin{equation}
\label{LT-SA}
\sum_{\la\in\si_\dd(J)} \dist\big(\la,\E\big)^{p-\f12}
\leq L_{p,\,\E} \sum_{n=-\infty}^\infty |\de a_n|^{p}+|\de b_n|^{p}, \quad p\geq 1,
\end{equation}
with some explicit constants $L_{p,\,\E}$ independent of $J$. Here, $\si_\dd(J)$ is the discrete spectrum of $J$. It was also shown in \cite{HS02} that the inequality
is false for $p<1$.

More recently, \eqref{LT-SA} was extended to selfadjoint perturbations of periodic and almost periodic Jacobi matrices with absolutely
continuous finite gap spectrum \cite{HS08,DKS10,FS11,CZ16}. When $\E$ is a finite gap set (i.e., a finite union of disjoint, compact intervals), the role of $J_0$ as a natural background operator 
is taken over by the so-called isospectral torus, denoted $\cc T_\E$. See, e.g., \cite{CSZ1,JSC2,SY97,Re11,Si11} for a deeper discussion of this object.
For $J'\in\cc T_\E$ and a compact selfadjoint perturbation $J=J'+\de J$, Frank and Simon \cite{FS11} proved \eqref{LT-SA} for $p=1$ while the case of $p>1$ is established
in \cite{CZ16}. The constant $L_{p,\,\E}$ is now independent of $J$ and $J'$ and only depends on $p$ and the underlying set $\E$.

As alluded to above, there is a general result of Kato \cite{Ka87} which applies to compact selfadjoint perturbations of arbitrary bounded selfadjoint operators. Specialized to
the case of perturbations of Jacobi matrices from $\cc T_\E$, it states that
\begin{align} \lb{Kato-SA}
\sum_{\la\in\si_\dd(J)} \dist\big(\la,\E\big)^{p} \leq \|\de J\|_p^p \leq \sum_{n=-\infty}^\infty (4|\de a_n|+|\de b_n|)^p, \quad p\geq1,
\end{align}
where $\|\cdot\|_p$ denotes the Schatten norm. In contrast to the Lieb--Thirring bounds, the power on the eigenvalues in \eqref{Kato-SA} is the same as on the perturbation
and so is larger than the power on the eigenvalues in \eqref{LT-SA} by $1/2$.
Kato's inequality is optimal for perturbations with large sup norm. On the other hand, the Lieb--Thirring bound with $p=1$ is optimal for perturbations with small sup
norm (cf.~\cite{HS02}). A fact that seemingly went unnoticed is that one can combine \eqref{LT-SA} and \eqref{Kato-SA} into one ultimate inequality which is optimal for
both large and small perturbations (at least when $p=1$). This inequality takes the form
\begin{align} \lb{Ultimate-SA}
\sum_{\la\in\si_\dd(J)} \dist\big(\la,\E\big)^{p-\f12}(1+|\la|)^{\f12} \leq C_{p,\,\E}\sum_{n=-\infty}^\infty |\de a_n|^p+|\de b_n|^p, \quad p\geq1,
\end{align}
where the constant $C_{p,\,\E}$ is independent of $J$ and $J'$, $J=J'+\de J$, $\de J$ is compact, $J'\in\cc T_\E$, and $\E$ is a finite gap set.

In recent years, several results have also been established for non-selfadjoint perturbations of selfadjoint Jacobi matrices \cite{BGK09,Ha11,HK11,GK12,Ha13}.
For compact non-selfadjoint perturbations $J=J_0+\de J$ of the free Jacobi matrix $J_0$, a near generalization (with an extra $\eps$) of the Lieb--Thirring bound
\eqref{LT-SA} was obtained by Hansmann and Katriel \cite{HK11} using the complex analytic approach developed in \cite{BGK09}.
Their non-selfadjoint version of the Lieb--Thirring inequalities takes the following form:
{\it For every $0<\eps<1$,
\begin{align}
\label{LT0-NSA}
\sum_{z\in\si_\dd(J)} \f{\dist\big(z,[-2,2]\big)^{p+\eps}}{|z^2-4|^{\f12}}
\leq L_{p,\,\eps} \sum_{n=-\infty}^\infty |\de a_n|^{p}+|\de b_n|^{p}+|\de c_n|^{p}, \quad p\geq1,
\end{align}
where the eigenvalues are repeated according to their algebraic multiplicity and the constant $L_{p,\,\eps}$ is independent of $J$}.
Whether or not this inequality continues to hold for $\eps=0$ is an open problem.

For non-selfadjoint perturbations of Jacobi matrices from finite gap isospectral tori $\cc T_\E$, 
an eigenvalue power bound was first obtained by Golinskii and Kupin in \cite{GK12}.
Shortly thereafter, this bound was superseded by a generalization of Kato's inequality to non-selfadjoint perturbations of arbitrary bounded selfadjoint operators (see Hansmann \cite{Ha13}). In the special case of a compact non-selfadjoint perturbation $J=J'+\de J$ of $J'\in\cc T_\E$, Hansmann's result reads
\begin{align}
\label{Kato-NSA}
\sum_{z\in\si_\dd(J)} \dist\big(z,\E\big)^{p}
\leq K_{p} \sum_{n=-\infty}^\infty |\de a_n|^{p}+|\de b_n|^{p}+|\de c_n|^{p}, \quad p>1,
\end{align}
where the eigenvalues are repeated according to their algebraic multiplicity and $K_p$ is a universal constant that depends only on $p$.

The purpose of the present article is to generalize the Lieb--Thirring bound \eqref{Ultimate-SA} to the case of compact non-selfadjoint perturbations $J=J'+\de J$ of Jacobi matrices $J'$ from finite gap isospectral tori $\cc T_\E$. Let $\pd\E$ denote the set of endpoints of the intervals that form $\E$. Then our main result can be formulated as follows:
{\it For every $p\geq1$ and any $\eps>0$,
\begin{align}
\label{LT-NSA}
\sum_{z\in\si_\dd(J)} \f{\dist\big(z,\E\big)^{p+\eps}(1+|z|)^{\f{1-3\eps}{2}}}{\dist(z,\pd\E)^{\f12}}
\leq L_{\eps,\,p,\,\E} \sum_{n=-\infty}^\infty |\de a_n|^{p}+|\de b_n|^{p}+|\de c_n|^{p},
\end{align}
where the eigenvalues are repeated according to their algebraic multiplicity and the constant $L_{\eps,\,p,\,\E}$ is independent of $J'$ and $J$.}
We note that  for the eigenvalues that accumulate to $\pd\E$, the inequality \eqref{LT-NSA} gives a qualitatively better estimate than \eqref{Kato-NSA}.
We also point out that \eqref{LT-NSA} is new even for perturbations of the free Jacobi matrix $J_0$ since, unlike \eqref{LT0-NSA}, it is nearly optimal not
only for small but also for large perturbations. As with \eqref{LT0-NSA}, it is an open problem whether or not \eqref{LT-NSA} remains true for $\eps=0$.

\section{Schatten Norm Estimates}
\label{sec2}

In this section we establish the fundamental estimates that are needed to prove our main result, Theorem \ref{FGap-LT-Thm}.
Throughout, $\cc S_p$ will denote the Schatten class and $\|\cdot\|_p$ the corresponding Schatten norm for $p\geq 1$.
To clarify our application of complex interpolation, we occasionally use $\|\cdot\|_\infty$ to denote the operator norm.

\begin{theorem} \lb{TNE-Thm}
Suppose $J'$ is a selfadjoint Jacobi matrix and $D\geq0$ is a diagonal matrix of Schatten class $\cc S_p$ for some $p\geq1$.
Denote by $d\rho_n$ the spectral measure of $(J',\de_n)$, that is, the measure in the Herglotz representation of the $n$th diagonal entry of $(J'-z)^{-1}$,
\begin{align} \lb{drho0}
\bigl\langle\de_n,(J'-z)^{-1}\de_n\bigr\rangle = \int_{\si(J')} \f{d\rho_n(t)}{t-z}, \quad z\in\bb C\bsn\si(J').
\end{align}
Then
\begin{align} \lb{TNE0}
\|D^{1/2}(J'-z)^{-1}D^{1/2}\|_p^p \leq \f{\sqrt2\,\|D\|_p^p}{\dist(z,\si(J'))^{p-1}}\,\sup_{n\in\bb Z} \int_{\si(J')} \f{d\rho_n(t)}{|t-z|}
\end{align}
for $z\in\bb C\bsn\si(J')$.
\end{theorem}
\begin{proof}
We consider first the case $p=1$. Let $\{P(t)\}_{t\in\bb R}$ denote the projection-valued spectral family of the selfadjoint operator $J'$. Recall that for any measurable and bounded function $f$ on $\si(J')$, the bounded operator $f(J')$ is given by the functional calculus,
\begin{equation} \lb{FnCalc}
f(J')=\int_{\si(J')}f(t)dP(t).
\end{equation}
Taking $f(t)=1/(t-z)$ in \eqref{FnCalc}, substituting into \eqref{drho0}, and recalling that the measure in the Herglotz representation is unique yield
\begin{equation} \lb{drho1}
\<\de_n,dP(t)\de_n\>=d\rho_n(t).
\end{equation}
Applying \eqref{FnCalc} to $f(t)=1/|t-z|$ and using \eqref{drho1} then imply
\begin{align} \lb{abs-int}
\<\de_n,|J'-z|^{-1}\de_n\> = \int_{\si(J')}\f{d\rho_n(t)}{|t-z|}, \quad z\in\bb C\bsn\si(J').
\end{align}
We also note that if, in addition, the function $f(t)$ in \eqref{FnCalc} is nonnegative, then $f(J')$ is a bounded, selfadjoint, and nonnegative operator.

Fix $z\in\bb C\bsn\si(J')$. In the following we assume without loss of generality that $\Im(z)\geq0$. Define the nonnegative functions
\begin{align}
f(t)=\Im\Big(\f{1}{t-z}\Big), \quad
f_+(t)&=\Re\Big(\f{1}{t-z}\Big)\chi_{(\Re(z),\infty)}(t), \notag \\
f_-(t)&=-\Re\Big(\f{1}{t-z}\Big)\chi_{(-\infty,\Re(z)]}(t),
\end{align}
and note that
\begin{align}
& f_+(t) - f_-(t) + if(t) = \f{1}{t-z},
\\
& f_+(t) + f_-(t) + f(t) = \Big|\Re\Big(\f{1}{t-z}\Big)\Big|+\Big|\Im\Big(\f{1}{t-z}\Big)\Big| \leq \f{\sqrt2}{|t-z|},
\quad t\in\bb R.
\end{align}
Then we have $f(J')\geq0$, $f_\pm(J')\geq0$, and
\begin{align}
& (J'-z)^{-1}=f_+(J')-f_-(J')+if(J'),   \lb{TNE1}
\\
& f_+(J') + f_-(J') + f(J') \leq \sqrt2\,|J'-z|^{-1}.   \lb{TNE2}
\end{align}
Using \eqref{TNE1}, the triangle inequality, and the fact that for nonnegative operators the trace norm coincides with the trace, we obtain the estimate
\begin{align}
&\|D^{1/2}(J'-z)^{-1}D^{1/2}\|_1    \no
\\
&\quad \leq  \no
\|D^{1/2}f_+(J')D^{1/2}\|_1 + \|D^{1/2}f_-(J')D^{1/2}\|_1 + \|D^{1/2}f(J')D^{1/2}\|_1
\\
&\quad =   \lb{TNE3}
\tr\big(D^{1/2}f_+(J')D^{1/2}\big) + \tr\big(D^{1/2}f_-(J')D^{1/2}\big) + \tr\big(D^{1/2}f(J')D^{1/2}\big).
\end{align}
Let $D_{n,n}$ denote the diagonal entries of $D$. Since $D$ is nonnegative and diagonal, we have
\begin{equation}
\sum_{n\in\bb Z} D_{n,n}=\tr(D)=\|D\|_1.
\end{equation}
Hence, by linearity of the trace it follows from \eqref{TNE3}, \eqref{TNE2}, and \eqref{abs-int} that
\begin{align}
&\|D^{1/2}(J'-z)^{-1}D^{1/2}\|_1  \leq
\tr\big(D^{1/2}[f_+(J')+f_-(J')+f(J')]D^{1/2}\big)  \no
\\
& \qquad \qquad \qquad \quad \leq    \no
\sqrt2\,\tr\big(D^{1/2}|J'-z|^{-1}D^{1/2}\big) \\
& \qquad \qquad \qquad \quad = \no
\sqrt2\sum_{n\in\bb Z} D_{n,n} \int_{\si(J')}\f{d\rho_n(t)}{|t-z|}
\\
& \qquad \qquad \qquad \quad \leq \lb{TNE4}
\sqrt2\,\|D\|_1\,\sup_{n\in\bb Z} \int_{\si(J')} \f{d\rho_n(t)}{|t-z|}.
\end{align}
This is exactly the case $p=1$ in \eqref{TNE0}.

To obtain \eqref{TNE0} for $p>1$, we use complex interpolation. Define the map
\begin{equation}
\ze\mapsto T(\ze)=D^{\ze p/2}|J'-z|^{-1}D^{\ze p/2}
\end{equation}
from the strip $0\leq\Re(\ze)\leq1$ into the space of bounded operators. Then for any $u,v\in\ell^2(\bb Z)$, the scalar function
\begin{equation}
\ze\mapsto \<u,T(\ze)v\>
\end{equation}
is continuous on the strip $0\leq\Re(\ze)\leq1$, analytic in its interior, and bounded. In addition, since $\|D^{iy}\|\leq1$ and
\begin{equation}
 D^{x+iy}=D^{iy}D^x=D^xD^{iy} \; \mbox{ for all $x\geq0$, $y\in\bb R$},
\end{equation}
it follows that
\begin{equation}
\|T(iy)\|_\infty \leq \||J'-z|^{-1}\| \leq \f{1}{\dist(z,\si(J'))},
\quad y\in\bb R,
\end{equation}
and by \cite[Theorem~2.7]{Si05} and \eqref{TNE4},
\begin{align}
\|T(1+iy)\|_1 & \leq \|D^{p/2}|J'-z|^{-1}D^{p/2}\|_1 \no \\
&\leq \sqrt2\,\|D^p\|_1\,\sup_{n\in\bb Z} \int_{\si(J')} \f{d\rho_n(t)}{|t-z|},
\quad y\in\bb R.
\end{align}
Thus, by the complex interpolation theorem (see \cite[Theorem~2.9]{Si05}, \cite[Theorem~III.13.1]{GK69}), we have
\begin{align}
\|T(x)\|_{1/x} \leq \sup_{y\in\bb R}\|T(iy)\|_\infty^{1-x} \sup_{y\in\bb R}\|T(1+iy)\|_1^{x}, \quad 0<x<1.
\end{align}
Taking $x=1/p$, raising both sides to the power $p$, and noting that $T(1/p)=D^{1/2}(J'-z)^{-1}D^{1/2}$ and $\|D^p\|_1=\|D\|_p^p$ finally yields \eqref{TNE0}.
\end{proof}

In what follows, $\E\subset\bb R$ will denote a finite gap set, that is,
\begin{align} \lb{SetE}
\E = \bigcup_{n=1}^{N} [\al_n,\be_n],\quad \al_1<\be_1<\al_2<\cdots<\al_N<\be_N,\quad N\geq1,
\end{align}
and $\pd\E$ will be the set of endpoints of $\E$, that is,
\begin{equation}
\pd\E=\{\al_n,\be_n\}_{n=1}^N.
\end{equation}
For a probability measure $d\rho$ supported on $\E$, we define the associated $m$-function by
\begin{equation}
m(z)=\int_\E\f{d\rho(t)}{t-z}, \quad z\in\bb C\bsn\E.
\end{equation}
The measure $d\rho$ is called {\it reflectionless} (on $\E$) if
\begin{equation}
\Re[m(x+i0)]=0 \; \mbox{ for a.e.\ $x\in\E$}.
\end{equation}
Reflectionless measures appear prominently in spectral theory of finite and infinite gap Jacobi matrices (see, e.g., \cite{JSC2,Re11,Si11,SY97}).
In particular, the isospectral torus associated to 
$\E$ is the set of all Jacobi matrices $J'$ that are reflectionless on $\E$
(i.e., the spectral measure of $(J',\de_n)$ is reflectionless for every $n\in\bb Z$) and for which $\si(J')=\E$.
It is well known (see for example \cite{SY97}) that $d\rho$ is a reflectionless probability measure on $\E$ if and only if $m(z)$ is of the form
\begin{align}
\label{ReflM}
m(z) = \f{-1}{\sqrt{(z-\be_N)(z-\al_1)}} \prod_{j=1}^{N-1}\f{z-\ga_j}{\sqrt{(z-\be_j)(z-\al_{j+1})}},
\end{align}
for some $\ga_j\in[\be_j,\al_{j+1}]$, $j=1,\dots,N-1$.

We now provide an estimate for the variant of the $m$-function for $d\rho_n$ that appear in Theorem \ref{TNE-Thm}.

\begin{theorem} \lb{ReflEst-Thm}
Let $\E\subset\bb R$ be a finite gap set and suppose $d\rho$ is a reflectionless probability measure on $\E$. Then for every $p>1$,
\begin{align} \lb{ReflEst1}
\int_\E \f{d\rho(t)}{|t-z|^p} \leq \f{K_{p,\,\E}}{\dist(z,\E)^{p-1}\dist(z,\pd\E)^{\f12}(1+|z|)^{\f12}}, \quad z\in\bb C\bsn\E,
\end{align}
where the constant $K_{p,\,\E}$ is independent of $d\rho$. In addition, for every $\eps>0$,
\begin{align} \lb{ReflEst2}
\int_\E \f{d\rho(t)}{|t-z|} \leq \f{K_{\eps,\,\E}}{\dist(z,\E)^{\eps}\,\dist(z,\pd\E)^{\f12}(1+|z|)^{\f12-\eps}}, \quad z\in\bb C\bsn\E,
\end{align}
where the constant $K_{\eps,\,\E}$ is independent of $d\rho$.
\end{theorem}
\begin{proof}
Denote the bands of $\E$ as in \eqref{SetE}. Since $d\rho$ is reflectionless on a finite gap set, it follows from the Stieltjes inversion formula and \eqref{ReflM}
that $d\rho$ is absolutely continuous with density given by
\begin{align} \lb{ReflEst3}
w(t)=\frac{1}{\pi}\Im[m(t+i0)]= \f{1/\pi}{\sqrt{|t-\be_N|\,|t-\al_1|}\,} \prod_{j=1}^{N-1}\f{|t-\ga_j|}{\sqrt{|t-\be_j|\,|t-\al_{j+1}|}\,}
\end{align}
for some $\ga_j\in[\be_j,\al_{j+1}]$, $j=1,\dots,N-1$. Fix $1\leq k\leq N$ and rearrange the terms in \eqref{ReflEst3} as follows
\begin{align} \lb{ReflEst4}
w(t) = \f{1/\pi}{\sqrt{|t-\al_k|\,|t-\be_k|}\,} \prod_{j=1}^{k-1}\f{|t-\ga_j|}{\sqrt{|t-\al_j|\,|t-\be_j|}\,}
\prod_{j=k+1}^{N}\f{|t-\ga_{j-1}|}{\sqrt{|t-\al_j|\,|t-\be_j|}\,}.
\end{align}
Since $\al_j<\be_j\leq\ga_j\leq\al_{j+1}<\be_{j+1}$ for $j=1,\dots,N-1$,
the two products in \eqref{ReflEst4} are at most $1$ for every $t\in[\al_k,\be_k]$ and thus
\begin{align} \lb{ReflEst4a}
w(t) \leq \f{1/\pi}{\sqrt{|t-\al_k|\,|t-\be_k|}\,}, \quad t\in[\al_k,\be_k].
\end{align}
Applying this estimate 
for the individual bands of $\E$ implies that
\begin{align} \lb{ReflEst5}
\int_\E \f{d\rho(t)}{|t-z|^p} \leq \f{1}{\pi}\sum_{k=1}^N \int_{\al_k}^{\be_k} \f{1}{|t-z|^p}\f{dt}{\sqrt{|t-\al_k|\,|t-\be_k|}\,}, \quad z\in\bb C\bsn\E.
\end{align}
By \cite[Lemma~11]{HK11}, for $p>1$, each integral in the sum can be estimated by
\begin{align} \lb{ReflEst6}
\int_{\al_k}^{\be_k} \f{1}{|t-z|^p}\f{dt}{\sqrt{|t-\al_k|\,|t-\be_k|}\,} \leq
\f{K_p}{\dist(z,[\al_k,\be_k])^{p-1}\sqrt{|z-\al_k|\,|z-\be_k|}\,}.
\end{align}
Since the function $z\mapsto\dist(z,\pd\E)(1+|z|)/|z-\al_k||z-\be_k|$ is continuous on $\bb C\bsn\{\al_k,\be_k\}$ and bounded near $\al_k$, $\be_k$, and $\infty$,
it is bounded on $\bb C\bsn\E$, and therefore
\begin{align} \lb{ReflEst7}
\int_{\al_k}^{\be_k} \f{1}{|t-z|^p}\f{dt}{\sqrt{|t-\al_k|\,|t-\be_k|}\,} \leq
\f{K_{p,\,\E}}{\dist(z,\E)^{p-1}\dist(z,\pd\E)^{\f12}(1+|z|)^{\f12}}, \quad p>1.
\end{align}
Combining \eqref{ReflEst7} with \eqref{ReflEst5} yields \eqref{ReflEst1}.

In order to obtain \eqref{ReflEst2}, note that since $\E$ is a bounded set we have the trivial bound
\begin{align}
\f{|t-z|}{1+|z|} \leq \f{|t|+|z|}{1+|z|} \leq K_\E, \quad t\in\E,\quad z\in\bb C\bsn\E.
\end{align}
This inequality yields the estimate
\begin{align}
\int_\E \f{d\rho(t)}{|t-z|} &= \int_\E \f{|t-z|^\eps d\rho(t)}{|t-z|^{1+\eps}} \no \\
&\leq K_\E^{\eps}(1+|z|)^\eps\int_\E \f{d\rho(t)}{|t-z|^{1+\eps}}, \quad z\in\bb C\bsn\E
\end{align}
and hence \eqref{ReflEst2} follows from \eqref{ReflEst1}.
\end{proof}

\section{Lieb--Thirring Bounds}
\label{sec3}

We start this section by recalling some results on the distribution of zeros of analytic functions with restricted growth towards the boundary of the domain of analyticity.
Let $a_+$ denote the maximum of $a$ and $0$. The following theorem for analytic functions on the unit disk is an alternative form of the extension
\cite[Theorem~4]{HK11} of the earlier result \cite[Theorem~0.2]{BGK09}.

\begin{theorem} \lb{DoZ-D-Thm}
Let $S\subset\pd\bb D$ be a finite collection of points and suppose $h(z)$ is an analytic function on $\bb D$ such that $|h(0)|=1$ and for some $K,\al,\be,\ga\geq0$,
\begin{align}
\log|h(z)| \leq \f{K|z|^\ga}{(1-|z|)^\al\,\dist(z,S)^\be}, \quad z\in\bb D.
\end{align}
Then for every $\eps>0$, there exists a constant $C_{\al,\be,\ga,\eps}$ independent of $h(z)$ such that the zeros of $h(z)$ satisfy
\begin{align}
\sum_{z\in\bb D,\, h(z)=0} \f{(1-|z|)^{\al+1+\eps}\, \dist(z,S)^{(\be-1+\eps)_+}}{|z|^{(\ga-\eps)_+}} \leq C_{\al,\be,\ga,\eps} K,
\end{align}
where each zero is repeated according to its multiplicity.
\end{theorem}

In \cite{GK12}, an analogous result on the distribution of zeros of analytic functions on $\Om=\overline{\bb C}\bsn\E$ was obtained via a reduction to the unit disk case.
For our purposes we will need the following extension of \cite[Theorem~0.1]{GK12} where an additional decay assumption at infinity is imposed in exchange for a
stronger conclusion. The extension follows from the reduction to the unit disk case developed in \cite{GK12} combined with the above version (Theorem~\ref{DoZ-D-Thm})
of the unit disk result. We omit the proof as it is a straightforward modification of the one presented in \cite{GK12}.

\begin{theorem} \lb{DoZ-Om-Thm}
Let $\E\subset\bb R$ be a finite gap set and suppose $f(z)$ is an analytic function on $\Om=\overline{\bb C}\bsn\E$ such that $|f(\infty)|=1$ and for some $K,p,q,r\geq0$,
\begin{align}
\log|f(z)| \leq \f{K}{\dist(z,\E)^p\,\dist(z,\pd\E)^q(1+|z|)^r}, \quad z\in\Om.
\end{align}
Then for every $\eps>0$, there exists a constant $C_{p,q,r,\eps}$ independent of $f(z)$ such that the zeros of $f(z)$ satisfy
\begin{align}
\sum_{z\in\Om,\, f(z)=0} \dist(z,\E)^{p'}\dist(z,\pd\E)^{q'}(1+|z|)^{r'} \leq C_{p,q,r,\eps} K,
\end{align}
where $p' = p+1+\eps$, $q' = \f12[(p+2q-1+\eps)_+ - p']$, $r' = (p+q+r-\eps)_+ - p' - q'$, and each zero is repeated according to its multiplicity.
\end{theorem}

We are now ready to present our finite gap version of the Lieb--Thirring inequalities for non-selfadjoint perturbations of Jacobi matrices from the isospectral torus $\cc T_\E$.

\begin{theorem} \lb{FGap-LT-Thm}
Let $\E\subset\bb R$ be a finite gap set and suppose $J$, $J'$ are two-sided Jacobi matrices such that $J'\in\cc T_\E$ and $J=J'+\de J$ is a compact perturbation of $J'$.
Then for every $p\geq1$ and any $\eps>0$,
\begin{align} \lb{FGap-LT}
\sum_{z\in\si_d(J)} \f{\dist(z,\E)^{p+\eps}(1+|z|)^{\f{1-3\eps}{2}}}{\dist(z,\pd\E)^{\f12}}
\leq L_{\eps,\,p,\,\E}\sum_{n=-\infty}^{\infty}|\de a_n|^p+|\de b_n|^p+|\de c_n|^p,
\end{align}
where the eigenvalues are repeated according to their algebraic multiplicity and the constant $L_{\eps,\,p,\,\E}$ is independent of $J$ and $J'$.
\end{theorem}
\begin{proof}
Suppose that $\de J\in \cc S_p$ for some $p\geq 1$ and define $D\geq0$ to be the diagonal matrix with the entries
\begin{align}
D_{n,n}=\max\bigl\{|\de a_{n-1}|,|\de a_n|,|\de b_n|,|\de c_{n-1}|,|\de c_n|\bigr\}, \quad n\in\bb Z.
\end{align}
A straightforward verification shows that $\de J = D^{1/2}BD^{1/2}$, where $B$ is a bounded tridiagonal matrix whose entries lie in the unit disk.
This in particular means that $\|B\|\leq3$.
Define
\begin{align}
f(z) = \sideset{}{_{\lceil p\rceil}}\det\bigl(I+D^{1/2}(J'-z)^{-1}D^{1/2}B\bigr),
\end{align}
where $\lceil p\rceil$ is the smallest integer $\geq p$. This regularized perturbation determinant is analytic on $\Om=\overline{\bb C}\bsn\E$ (see, e.g., \cite[Chapter IV. \S 3]{GK69}). More importantly, the zeros of $f$ coincide with the discrete eigenvalues of $J$ and the multiplicity of the zeros matches the algebraic multiplicity of the corresponding eigenvalues (see \cite{HK11} and \cite[Appendix C]{F} for a proof).
By \cite[Lemma~XI.9.22(d)]{DS63}, we have
\begin{align}
\log|f(z)| \leq K_p\|D^{1/2}(J'-z)^{-1}D^{1/2}B\|_p^p
\leq 3K_p\|D^{1/2}(J'-z)^{-1}D^{1/2}\|_p^p
\end{align}
for some constant $K_p$. It thus follows from Theorems \ref{TNE-Thm} and \ref{ReflEst-Thm} (with $\eps/2$ instead of $\eps$) that
\begin{align}
\log|f(z)| \leq \f{K_{\eps,p,\E}\,\|D\|_p^p} {\dist(z,\E)^{p+\f{\eps}{2}-1} \dist(z,\pd\E)^{\f12}(1+|z|)^{\f{1-\eps}{2}}}, \quad z\in\Om.
\end{align}
Applying Theorem~\ref{DoZ-Om-Thm} (with $\eps/2$ instead of $\eps$) and noting that ${(1-3\eps)/2}\leq r'$ and $D_{n,n}^p\leq|\de a_{n-1}|^p+|\de a_n|^p+|\de b_n|^p+|\de c_{n-1}|^p+|\de c_n|^p$ then yield \eqref{FGap-LT}.
\end{proof}



\end{document}